\let\nofiles\relax

\documentclass[twoside]{article}

\usepackage{subfigure}
\usepackage{mathrsfs}
\usepackage[sumlimits]{amsmath}
\usepackage{amsthm,amssymb,amsfonts}
\usepackage{graphicx,color}
\usepackage{algorithm}
\usepackage{algpseudocode}
\usepackage{titlesec}
\usepackage{float}
\usepackage{caption}
\usepackage{extarrows}

\usepackage[
linktocpage=true,
colorlinks=true,
linkcolor=blue,
citecolor=blue,
urlcolor=blue
]{hyperref}
\usepackage[numbers,sort&compress]{natbib}

\allowdisplaybreaks

\catcode`@=11
\long\def\@makefntext#1{\noindent #1}
\newskip\tabcentering\tabcentering=1000pt plus 1000pt minus 1000pt

\def\MCH#1#2{\setbox0=\hbox{\raise#1\hbox{#2}}\smash{\box0}}

\def\evenhead{}
\def\oddhead{}
\def\@evenhead{\hbox to\textwidth{\footnotesize\rm\thepage\hfill{\small\it\evenhead}}}
\def\@oddhead{\hbox to\textwidth{\footnotesize{\small\it\oddhead}\hfill\footnotesize\rm\thepage}}
\def\@evenfoot{}
\def\@oddfoot{}

\catcode`@=12
\font\bb=cmbx8%
\font\small=cmr8

\font\st=cmbx10 scaled\magstep2%
\def\title#1{{\noindent\st#1}}%
\def\author#1{\vspace*{0.3 true cm}{\noindent\bf #1}}%
\def\abstract#1{\vspace*{1 true cm}{\noindent{\footnotesize{\bb Abstract}\hspace{4true mm}#1}}}%

\newtheorem{definition}{Definition}[section]

\newtheorem{theorem}[definition]{Theorem}

\newtheorem*{remark*}{Remark}

\newtheoremstyle{noparens}%
{}{}%
{\itshape}{}%
{\bfseries}{\bf .}%
{ }%
{\thmname{#1}\thmnumber{ #2}\mdseries\thmnote{#3}}

\theoremstyle{noparens}

\theoremstyle{definition}

\numberwithin{equation}{section}
\numberwithin{figure}{section}

\titleformat*{\section}{\large \bfseries}
\titleformat*{\subsection}{\bfseries}
\titleformat*{\subsubsection}{\normalsize \bfseries}

\begin{document}
%
%
%
%



\title{Polynomial Stability of a Type II Porous Thermoelastic System with Local Memory Damping}%

\author{SUN Ya-nan, ZHANG Qiong}


\vskip10pt


\vskip-10pt

\begin{abstract}
	This paper studies the asymptotic behavior of a one-dimensional Type II porous thermoelastic system with a conservative porous structure and local memory damping applied to the elastic component. Using frequency domain resolvent estimates, we prove  polynomial decay of the associated semigroup. Our results clarify the effect of local memory damping and provide a unified framework for partially damped porous thermoelastic systems.
\end{abstract}


\section{Introduction}
Porous elastic materials, also referred to as elastic materials with voids, constitute an important class of media characterized by the presence of microstructures such as pores or cavities.
These materials were systematically introduced in the late 1970s and early 1980s by Cowin and Nunziato \cite{Cowin1,Cowin2,Cowin3}, who developed a continuum theory in which the mass density depends not only on the elastic deformation but also on the volume fraction of the voids.
When thermal effects are incorporated into porous elastic systems, the classical approach relies on Fourier's law of heat conduction, which leads to a parabolic heat equation.
However, this framework predicts an infinite speed of thermal propagation, a feature that contradicts experimental observations.
To overcome this drawback, Green and Naghdi \cite{Green}  proposed three new thermoelastic theories based on the introduction of new constituent variables for the description of heat.
Of particular relevance here is Type II theory,
which excludes the temperature gradient and  leads to  a conservative hyperbolic
heat equation, often called ``without energy dissipation''.
The coupling of  elasticity with voids  and Green-Naghdi Type II  heat conduction has attracted considerable attention in recent years (see, for instance, \cite{Casas, Pamplona2009, Quintanilla2003, Zhang} and references therein), as such  systems exhibit rich interactions among the displacement, volume fraction, and thermal variables.

Compared with the global dissipation, localized dissipation, where effects are confined to subdomains, typically leads to slower polynomial decay rates rather than exponential stability (see, e.g., \cite{LQSZ, Zhang}).
Another important dissipation mechanism arises from memory effects, which account for the hereditary behavior of materials.
For porous thermoelastic systems with global memory damping, general decay results have been established in \cite{Messaoudi2021,Quintanilla2011,2010}, where the energy decay rate is shown to be closely related to the properties of the memory kernel.
Motivated by these developments, the present paper is devoted to the study of a one-dimensional Type II porous thermoelastic system with a conservative porous structure and local memory damping acting on the elastic component.
The mathematical model under consideration is given by the following PDEs system:
\begin{align}\label{sys-new}
	\begin{cases} \displaystyle
		\rho u_{tt} =  \Big[ \mu  u_{x} + \mu^*(x) \int_0^{\infty} g_s(s)(u_{x}(t-s) -u_{x}(t)) ds\Big]_x+\gamma \phi_x -\beta\psi_{xt}  ,\\
		J\phi_{tt}=b\phi_{xx}+m\psi_{xx}-\xi\phi+d\psi_t-\gamma u_x,\\
		a\psi_{tt}=k\psi_{xx}+m\phi_{xx}-d\phi_t-\beta u_{xt},
	\end{cases}
\end{align}
with the following boundary conditions and  initial data:
\begin{align}\label{init} \begin{split}
		& u(0,t) = u(\pi,t) = \phi(0,t) = \phi(\pi,t) = \psi(0,t) = \psi(\pi,t) = 0,\;\;t>0,\\ &
		u(x,0) = u_0(x),\;u_t(x,0)=v_0(x),\;\;u(x,-s) = \nu(x,s),\;s>0,\\ &\phi(x,0)=\phi_0(x),\;
		\phi_t(x,0)=\varphi_0(x),\;
		\psi(x,0)=\psi_0(x),\;\psi_t(x,0)=\theta_0(x).
\end{split}\end{align}

\noindent
In \eqref{sys-new}-\eqref{init}, $u(x,t)$, $\phi(x,t)$, and $\psi(x,t)$ represent the displacement, the volume fraction, and the temperature, respectively. $t \geq 0$ denotes time, and the spatial variable $x$ lies in the interval $[0,\pi]$, corresponding to a rod of length $\pi$.
The physical parameters are assumed to satisfy
\begin{equation}
	\label{assume}
	\begin{gathered}
			\rho>0,\;\mu>0,\;\gamma>0,\;\beta\neq0,\;J>0,\;b>0,\;
		m\neq 0,\;\xi>0,\\
		\;d>0,\;a>0,\;k>0, \;
		\mu\xi>\gamma^2,\;bk>m^2.
	\end{gathered}
\end{equation}
The damping coefficient $\mu^*(\cdot)\in C^{1}(0,\pi)$ is non-negative and satisfies 
\begin{align}
	\label{hmu}
	\mu^*(\cdot) \ge \mu_0 >0 \;  \mbox{ on  }  \;   [0,\tau) , \;\;
	\mu^*(\cdot) =0  \;  \mbox{ on  }  \;  (\tau,\pi],  \;  \mbox{ where  }  \; \tau \in (0,\pi).
\end{align}

The memory kernel  $g(\cdot)$  satisfies the following hypotheses:
\begin{enumerate}
	\item[{\bf (H1)}]$g \in C^2[0,\infty)$, $ g_s  \in L^1(0,\infty)$, $g(s) >  0$, $g_s(s) < 0$, $g_{ss}(s) > 0$ for all $s\ge 0$, and  $g(\infty)=0$.
	\item[{\bf (H2)}]There exists a   constant $K>0$ such that $g_{ss}(s)+K g_s(s)  \ge 0$ for all $s\ge 0$.
\end{enumerate}

Conditions (H1)-(H2)  imply that the memory is strictly decreasing, the rate of memory loss is increasing.
The main objective of this work is to characterize the asymptotic behavior of solutions to   system \eqref{sys-new}-\eqref{init}.
By using the resolvent estimates, we prove polynomial stability of the associated $C_0$-semigroup.

The paper is structured as follows. Section 2 presents the well-posedness. Section 3 is devoted to the proof of the main stability result.
Throughout the paper, $|\cdot|$ and $\langle \cdot,\cdot\rangle$ denote the norm and inner product in $L^2(0,\pi)$, respectively.

\section{Well-posedness}
In this section, we  present well-posedness of the system.
First,
%
we introduce the history variable
$\eta(x,t,s) = u(x,t-s)  -  u(x,t),$ where $t,s>0.$
It is clear that
$\eta_t(x,t,s) + \eta_s(x,t,s)=- u_t(x,t) , $
and the first equation of
system \eqref{sys-new} can be reformulated as
\begin{align}\label{sys}
\rho u_{tt} =[\mu u_{x} + \mu^*(x)\zeta_{x}]_x +\gamma \phi_x -\beta\psi_{xt},\;\;\;  \mbox{   where}  \; \;\;\zeta(x,t) = \int_0^{\infty}g_s(s)\eta(x,t,s)ds,
\end{align}
We now introduce
the weighted memory space
$W=L^2(\mathbb R^+,-g_s;H_0^1(0,\pi)),$
and the associated energy space
$\mathcal H = H_{0}^1(0,\pi)\times L^2(0,\pi)\times W \times (H_0^1(0,\pi)\times L^2(0,\pi))^2$
endowed with the inner product
\begin{align*}
\langle U,U^*\rangle_{\mathcal H} = &\;  \int_0^{\pi} \Big(\rho v\overline{v^*}
+ J \varphi\overline{\varphi^*}
+ a \theta\overline{\theta^*}
+ \mu u_x\overline{u_x^*}
+b \phi_x\overline{\phi_x^*}
+\xi \phi\overline{\phi^*}
+ k \psi_x\overline{\psi_x^*}
+ \gamma(\phi\overline{u_x^*}
+u_x\overline{\phi^*}) \\ &
+m(\phi_x\overline{\psi_x^*} + \psi_x\overline{\phi_x^*})\Big)dx
-\int_0^{\infty}g_s(s)\int_0^{\pi} \mu^*(x) \eta_x\overline{\eta_x^*}dxds,
\end{align*}
for $U=(u,v,\eta(\cdot),\phi,\varphi,\psi,\theta)$ and $U^*=(u^*,v^*,\eta^*(\cdot),\phi^*,\varphi^*,\psi^*,\theta^*)\in \mathcal H$,

Define an unbounded linear operator $\mathcal A: \mathcal D(\mathcal A)\subset \mathcal H\to \mathcal H$ by
\begin{align*}
& \mathcal AU
= \Big(
v,\; {1\over\rho} ((\mu u_{x}+\mu^*(x)\zeta_{x})_x + \gamma\phi_x -\beta\theta_x ),\; -v-\eta_s(\cdot),\;
\varphi,\;\;
{1\over J} (b\phi_{xx}+m\psi_{xx}-\xi\phi +d\theta-\gamma u_x ),\;
\\ &
\hskip 1cm  \; \;
\theta,\;{1\over a}(k\psi_{xx}+m\phi_{xx}-d\varphi-\beta v_x)
\Big),
\\
&
\mathcal D(\mathcal A) = \big\{
U\in\mathcal H
\;\big| \;
v,\varphi,\theta\in H_0^1(0,\pi),\;
(\mu u_{x} +\mu^*(x)\zeta_{x})_x \in L^{2}(0,\pi) ,\;
\eta_s(\cdot)\in W,\;\eta(s=0) = 0,  \\ &\hskip 3cm
b\phi_{xx}+m \psi_{xx}\in L^{2}(0,\pi),\;
k\psi_{xx} + m \phi_{xx}\in L^{2}(0,\pi)\big\}.
\end{align*}
where $U = (u,v,\eta(\cdot),\phi,\varphi,\psi,\theta),\; \zeta = \int_0^{\infty}g_s(s)\eta(s)ds.$
Consequently, system \eqref{sys-new}-\eqref{init} can be written as the abstract evolution equation
$
{dU\over dt} = \mathcal A U,\;\forall\; t>0$ with $U(0)=U_0\doteq(u_0,v_0,\eta_0(\cdot),
\phi_0,\varphi_0,\psi_0,\theta_0)\in\mathcal H.$

The following theorem establishes the well-posedness of system \eqref{sys-new} and the proof of it is standard, we omit it here; see \cite{pazy,LQSZ} for details.
\begin{theorem}\label{th1}
Assume that the physical parameters and the coefficient function $\mu^*(\cdot)$ satisfy \eqref{assume}-\eqref{hmu}, and that the relaxation function $g$ satisfies (H1)-(H2). Then the operator $\mathcal A$ generates a $C_0$-semigroup $ e^{\mathcal A t }$ of contractions on $\mathcal H$, and the resolvent set satisfies $0\in\rho(\mathcal A)$.
\end{theorem}

\begin{theorem}\label{th2}
Let the conditions in Theorem \ref{th1} hold. Then, $i\mathbb{R}\subset \rho(\mathcal A)$, the resolvent of the operator $\mathcal A$.
\end{theorem}
\begin{proof}
By Theorem \ref{th1},   $0\in \rho(\mathcal A)$. For any  $0 \not=\lambda\in \mathbb{R},$
consider $
\label{reso}(i\lambda I-\mathcal{A})U=0,  $ where $ U=(u,v,\eta(\cdot),\phi,\varphi,\psi,\theta) \in D(\mathcal A).
$
This implies
$
v=i\lambda u,\;\;\varphi=i\lambda \phi,\;\; \theta =  i\lambda \psi,~
$
and that
\begin{align} \label{eq-v11}
	& -\rho\lambda^2 u-(\mu u_{x} + \mu^*(x)\zeta_x)_x - \gamma  \phi_x  +  i \beta \lambda \psi_x =0,&    x\in(0,\; \pi) , \vspace{2 mm}\\  \label{eq-v12}
	&-J\lambda^2 \phi-(b\phi_{xx} +m \psi_{xx}) +  \xi   \phi - i d \lambda \psi + \gamma u_x =0,&    x\in(0,\; \pi) , \vspace{2 mm} \\ \label{eq-v13}
	&-a\lambda^2 \psi- ( k\psi_{xx}   + m \phi_{xx})  + i d \lambda \phi + i \beta  \lambda u_x =0,&    x\in(0,\; \pi) ,\vspace{2 mm} \\
	\label{str-eta}
	& i\lambda \eta (\cdot)+v +\eta_s(\cdot)=0,&  s>0,\;   x\in(0,\; \tau) ,
\end{align}
It follows directly from \eqref{str-eta} that
$ \eta(s)= (e^{-i\lambda s} -1) u$.
Moreover, using
$Re\langle i\lambda I-\mathcal{A}U,U\rangle_{\mathcal H}=0$
together with
\begin{align}\label{diss}
	\Re \langle\mathcal A U,U\rangle_{\mathcal H}
	=-{1\over2}\int_0^{\infty}g_{ss}(s)\|\sqrt{\mu^*(x)}\eta_x\|^2 ds\leq 0,
\end{align}
we have that $\int_0^{\infty}g_{ss}(s)\|\sqrt{\mu^*(x)}\eta_x\|^2 ds=0.$
Combining this fact with assumption (H2) and the estimate $
	\inf \{
	\int_0^\infty |g_s(s)|\,  |1-e^{-i \lambda s}  |^2 ds \,| \, |\lambda|\ge \varepsilon>0 \}\ge \delta>0,$ which follows from Propositions 2.1 and 2.2 in \cite{ZQ-prop1}, we conclude that
\begin{align}
	\label{dis}
	u=v=\eta(\cdot)=0   \quad x\in(0, \tau).
\end{align}
Consequently, from \eqref{eq-v11}--\eqref{eq-v12} and \eqref{dis}, we deduce that \begin{align}
	\label{1d}
	\Big(-b +  i {\frac {m \gamma}{\beta \lambda}}\Big)\phi_{xx}  +\Big( -J\lambda^2+\xi  -  {\frac {\gamma d}{\beta}} \Big)\phi =0,  \; \;  x\in(0,\; \tau),
\end{align}
If $J\lambda^2 - \xi+ \frac{d \gamma}{\beta}\neq 0$, then the solution of \eqref{1d} is given by
$
\phi =c(e^{\sqrt{q_{1}}x} - e^{-\sqrt{q_{1}}x} ),$ where  $ x\in(0,\; \tau )$
and $q_{1} =  \frac{J\lambda^2 - \xi+ \frac{d \gamma}{\beta}}{-b + \frac{i m \gamma}{\beta \lambda}} \neq 0.$
Substituting this expression into \eqref{eq-v11} and \eqref{eq-v13}, we obtain
$
\Big[i  {\frac {a \gamma \lambda }{\beta}} +\Big( i {\frac {k \gamma}{\beta \lambda}}-m\Big) q_1 + id\lambda \Big] c (e^{\sqrt{q_{1}}x} - e^{-\sqrt{q_{1}}x} )=0$ for $  x\in(0,\; \tau ).
$
Since $
i  {{a \gamma \lambda }\over{\beta}} +\Big( i {{k \gamma}\over{\beta \lambda}}-m\Big) q_1 + id\lambda   \neq0,
$
it follows that  $\phi=0 $
on $(0,\;\tau).$ By \eqref{eq-v11}, we further infer that $\psi=0 $
on $(0,\;\tau).$
When $J\lambda^2 - \xi+ \frac{d \gamma}{\beta}= 0$,  one can also easily prove that  $\phi=\psi=0 $
on $(0,\;\tau).$
In summary, we have
$ u=v=\phi =\varphi=\psi=\theta=0 $ on $[0,\;\tau ).$
By the uniqueness of the solutions  to the ordinary differential equations \eqref{eq-v11}-\eqref{eq-v13}, we conclude that
$Ker(i\lambda I - {\mathcal A}) = \{0\}$.

Now we prove $Ran(i\lambda I - {\mathcal A}) = {\mathcal H}$  for every real $\lambda$.
Given $F=(f^{1},g^{1},h, f^{2},g^{2},f^{3},g^{3})  \in {\mathcal H},$ we seek to solve the equation $
(i \lambda  {\mathcal {I}}- {\mathcal {A}}) U    =F$ where $  U = ({u}, {v}, \eta(\cdot),  \phi, \varphi, {\psi}, {\theta}).$
Equivalently, $v = i\lambda u - f^{1} ,\;\; \varphi = i\lambda \phi-   f^{2},\;\; \theta  = i\lambda \psi-  f^{3} $ and
\begin{align}
	&   -\rho\lambda^2 u- ( \mu  u_{x} - \mu^*(x)\zeta_{x})_x -\gamma \phi_x
	+i\beta \lambda \psi_x     = \tilde{f}_1 = \rho g^{1}  +i\rho\lambda f^{1} + \beta f^3_x ,
	\label{s4}
	\vspace{2 mm} \\
	&
	-J\lambda^2 \phi+\gamma  u_x -b \phi_{xx}
	+\xi \phi  -m\psi_{xx}
	-i\lambda d \psi    =\tilde{f}_2 = J g^{2}  + i J \lambda f^2 - d f^3 ,
	\label{s5} \vspace{2 mm} \\
	& - a\lambda^2 \psi+ i\beta \lambda u_x -m \phi_{xx} +id\lambda\phi -k\psi_{xx}     =\tilde{f}_3 = a g^{3} +ia\lambda f^3+ \beta f^1_x +df^2,
	\label{s6}
	\vspace{2 mm} \\
	&\eta (s)= -\left({1-e^{-i\lambda  s}\over i\lambda }\right) (i\lambda  u  - f^1 ) + H (s),\;\;H (s) = \int_0^s e^{-i\lambda (s-\tau)}h (\tau)d\tau.\label{s61}
\end{align}
It is clear from \eqref{s6} that
\begin{align}
	\label{zeta}
	\zeta =  \int_0^{\infty}g_s(s)(e^{-i\lambda  s}-1)ds \, u + \tilde F, \;\; \tilde F= \int_0^{\infty}g_s(s)\Big( {1-e^{-i\lambda  s}\over i\lambda} f^1 +H(s)\Big) ds.
\end{align}
Under assumptions (H1)-(H2) and by Propositions 2.1-2.2 in \cite{ZQ-prop1}, we have
\begin{align}
	\label{hh}
	\inf\limits_{\big\{\lambda\in\mathbb{R}\:\big|\:|\lambda|\ge \varepsilon>0\big\}}
	\int_0^\infty |g_s(s)|\,  |1-e^{-i \lambda s}  |^2 ds \ge \delta>0,
	\;\;   \mbox{ and  }  \;\; \|H \|_W \le \|h \|_W   .
\end{align}
For any $ U_0=(u,\phi,\psi)  \in {\mathcal H}_0 = [H^1_0(0,\pi)]^3, $ we define
\begin{align*}
	{\mathcal A}_0 U _0   =
	\big(\big(\Phi u_{x})_x -\gamma \phi_x
	+i\beta \lambda \psi_x   ,\;
	\gamma  u_x -b \phi_{xx}
	+\xi \phi  -m\psi_{xx}
	-i\lambda d \psi
	,\;
	i\beta \lambda u_x -m \phi_{xx} +id\lambda\phi -k\psi_{xx}
	\big) ,
\end{align*}
where $\Phi = -\mu -\mu^*(x)\int_0^{\infty}g_s(s)(e^{-i\lambda  s}-1)ds.$
It is clear that  the equation ${\mathcal A}_0 U _0 = F_0, $ for all $ F_0 \in {\mathcal H}_0'$,  is equivalent
to $a(U_0,V) = \langle F_0,V\rangle_{{\mathcal H}_0', {\mathcal H}_0}$ for all $  V \in {\mathcal H}_0$, where   $a(\cdot,\cdot)$ is the following  continuous sesquilinear form on  ${\mathcal H}_0 \times {\mathcal H}_0$:
\begin{align*}
	a(U_0,\tilde{U}_0) =\;&
	\langle \Phi  u_x, \tilde{u}_x \rangle
	+\langle -\gamma \phi_x
	+i\beta \lambda \psi_x, \tilde{u} \rangle
	+ \langle\gamma  u_x-i\lambda d \psi, \tilde{\phi} \rangle + b\langle \phi_x, \tilde{\phi}_x \rangle
	+\xi\langle \phi , \tilde{\phi} \rangle+ m\langle\psi_x, \tilde{\phi}_x \rangle
	\\&
	+\langle i\beta \lambda u_x+id\lambda\phi, \tilde{\psi} \rangle
	+\langle m \phi_x  +k\psi_x  , \tilde{\psi}_x \rangle, \;\;\forall \; U_0=(u,\phi,\psi) ,  \;  \tilde{U}_0=(\tilde{u},\tilde{\phi},\tilde{\psi}) \in {\mathcal H}_0.
\end{align*}
A direct computation shows that there exists a constant $\tilde{a} >0$  such that
$Re\, a(U_0,U_0)  \ge \tilde{a} \|U_0\|_{{\mathcal H}_0}^2 $ for all $   U_0  \in {\mathcal H}_0.$
Then, by the Lax-Milgram's theorem, ${\mathcal A}_0$ is an
isomorphism of ${\mathcal H}_0$ onto ${\mathcal H}_0'$.
Therefore,   equations  \eqref{s4}-\eqref{s6} can be rewritten as
$  U_0 - \lambda^2{\mathcal A}_0^{-1}{\rm diag}\big(\rho, J, a \big) U_0
= {\mathcal A}_0^{-1} F_0,
$ where $F_0 = \big(\tilde{f}_1+(\mu^*(x)F_{x})_x,\;\tilde{f}_2,\; \tilde{f}_3\big)  \in {\mathcal H}_0'$.
If $U_0  =(u,\phi,\psi) \in Ker\, \big(I-  \lambda^2{\mathcal A}_0^{-1}{\rm diag}\big(\rho, J, a \big)\big) ,$ then
$
\lambda^2{\rm diag}\big(\rho, J, a \big) U_0 - {\mathcal A}_0  U_0 =0.$
It follows that
\begin{align}
	&   -\rho\lambda^2 u- (\Phi  u_x)_x -\gamma \phi_{x}
	+i\beta \lambda \psi_{x}     = 0,
	\label{s11}
	\\
	&
	-J\lambda^2 \phi+\gamma  u_{x} -b \phi_{xx}
	+\xi \phi  -m\psi_{xx}
	-i\lambda d \psi     =0,
	\label{s12}  \\
	& - a\lambda^2 \psi+ i\beta \lambda u_{x} -m \phi_{xx} +id\lambda\phi -k\psi_{xx}     =0.
	\label{s13}
\end{align}
Multiplying \eqref{s11} by $u$, \eqref{s12} by $\phi$, and \eqref{s13} by $\psi$, adding the results, and taking the imaginary parts,  we obtain
$
\mu^*(x)  u_x  \equiv 0.$
Using the same argument as in Step 1, we conclude that  $u,\; \phi,\; \psi \equiv0  $. This implies that   $Ker\, \big(I-  \lambda^2{\mathcal A}_0^{-1}{\rm diag}\big(\rho, J, a \big)\big) = \{0\}.$
Note that   ${\mathcal A}_0^{-1}$ is a compact operator on $[L^2(0,\pi)]^3$. Therefore, due to  Fredholm's
alternative,   equation $  U_0 - \lambda^2{\mathcal A}_0^{-1}{\rm diag}\big(\rho, J, a \big) U_0
= {\mathcal A}_0^{-1} F_0,$  
admits a unique solution $U_0 \in {\mathcal H}_0.$  In summary, we conclude that $ Ran \,(i \lambda  {\mathcal {I}}- {\mathcal {A}}) ={\mathcal H}$. The proof is complete.

\end{proof}

\section{Polynomial stability}
This section is devoted to the proof of the polynomial decay of the  system \eqref{sys-new}.
\begin{theorem}\label{main}
Assume that the hypotheses of Theorem \ref{th1} hold.
Then the $C_0$-semigroup $e^{t\mathcal A} $ associated with system \eqref{sys-new}
is polynomially stable with decay rate $t^{-5/8}$, i.e., there exists a constant $M>0$ such that
$$ \|e^{\mathcal A t}U\|_\mathcal{H} \le Mt^{-{5/8}}\|U \|_{\mathcal{D}(\mathcal{A})},  \quad
\forall \; U  \in  \mathcal{D}(\mathcal{A}),\; t\ge1. $$
\end{theorem}

\begin{proof}
According to the frequency-domain criterion of Borichev and Tomilov~\cite{BT},
it suffices to prove the existence of a constant $r>0$ such that
$\inf_{\lambda\in\mathbb R,\ \|U\|_{\mathcal H}=1}
|\lambda|^{p}\|(i\lambda I-\mathcal A)U\|_{\mathcal H}\ge r $ with $p=8/5.$
Assume, by contradiction, that the above estimate does not hold.
Then there exist a sequence $\{\lambda_n\}\subset\mathbb R^+$ with
$\lambda_n\to+\infty$ and a sequence
$U^n=(u^n,v^n,\eta^n,\phi^n,\varphi^n,\psi^n,\theta^n)\subset\mathcal D(\mathcal A)$
such that
\begin{equation}
\label{norm}\|U^n\|_{\mathcal H}=1,
\qquad
\lambda_n^{p}\|(i\lambda_n I-\mathcal A)U^n\|_{\mathcal H}=o(1),\; \; p=8/5.
\end{equation}
Let $T^n = \mu u^n_x+  \mu^*(x)\zeta^n_x, $ and $\zeta^n = \int_0^{\infty}g_s(s)\eta^n(s)ds$.
It follows that
\begin{subequations}
\begin{align}
	&\lambda_n^{p}[i\lambda_n u^n - v^n]= f^n= o(1),&& \mbox{in } H_0^1(0,\pi),\label{iR-u}\\
	&\lambda_n^{p}[i\rho\lambda_n v^n -T^n_{x} - \gamma\phi^n_{x} + \beta\theta_x^n]= o(1),&& \mbox{in } L^2(0,\pi),\label{iR-v}\\
	&\lambda_n^{p}[i\lambda_n\eta^n(\cdot)+v^n+\eta^n_s(\cdot)] = h^n(s) =  o(1),&& \mbox{in } W,\label{iR-eta}\\
	&\lambda_n^{p}[i\lambda_n\phi^n - \varphi^n] =o(1),&& \mbox{in } H_0^1(0,\pi),\label{iR-phi}\\
	&\lambda_n^{p}[iJ\lambda_n\varphi^n - b\phi^n_{xx} - m\psi^n_{xx}+\xi\phi^n - d\theta^n+\gamma u^n_x] = o(1),&& \mbox{in } L^2(0,\pi),\label{iR-varphi}\\
	&\lambda_n^{p}[i\lambda_n\psi^n - \theta^n ]= o(1),&& \mbox{in } H_0^1(0,\pi),\label{iR-psi}\\
	&\lambda_n^{p}[i a\lambda_n \theta^n - k\psi^n_{xx} - m\phi^n_{xx}+d\varphi^n +\beta v^n_x ]= o(1),&& \mbox{in } L^2(0,\pi).\label{iR-theta}
\end{align}
\end{subequations}
From \eqref{iR-u}, \eqref{iR-phi}, \eqref{iR-psi}, together with $\lambda_n\to \infty$, we obtain
\begin{align}\label{norm-o}   \lambda_n \|u^n\|, \;\; \lambda_n\|\phi^n\|,\;\; \lambda_n\|\psi^n\|,\;\;
\lambda_n^{-1} \|v^n_x\|,\;\;
\lambda_n^{-1} \|\varphi^n_x\|,\;\;
\lambda_n^{-1} \|\theta^n_x\|  =  {\mathcal O}(1).
\end{align}
Using assumption (H2), the dissipativeness of
$\mathcal A $,  the identity $\lambda_n^{p} \Re   \langle i\lambda_n U_n -\mathcal A U_n,U_n\rangle_{\mathcal H} \to 0  $, and the Cauchy-Schwarz inequality, it follows that
\begin{align}\label{iR-eta-o}
\|\eta^n(\cdot)\|_W,\;\;
\|\zeta_x^n\|_{L^2(0,\tau)} = \lambda_n^{-p/2}o(1).
\end{align}
By arguments similar to those used in \eqref{s6} and \eqref{hh}, we further deduce from \eqref{iR-u}, \eqref{iR-eta} and \eqref{iR-eta-o} that
\begin{align}\label{iR-u'-o}
\| \sqrt{\mu^*(x)} u_x^n\| = \lambda_n^{-p/2} o(1)
\end{align}
Moreover, the following   estimates also hold:
\begin{align}\label{bound}
\|\lambda_n^{-1}\phi^n_{xx}\|,\;\;
\|\lambda_n^{-1}\psi^n_{xx}\|,\;\;
\|\lambda_n^{-1}T^n_{x}\| = O(1).
\end{align}
Indeed, using $\|U^n\|_{\mathcal H}=1$ together with  and  \eqref{iR-u},  \eqref{iR-varphi}, \eqref{iR-theta}, we obtain
$
\|\lambda_n^{-1}(b\phi^n_{xx}+m\psi^n_{xx})\|
=O(1),\;
\|\lambda_n^{-1}(k\psi^n_{xx}+m\phi^n_{xx})\|
=O(1).
$
Since $bk > m^2$, the first two estimates in \eqref{bound} follow immediately.
Similarly, the last estimate is a direct consequence of \eqref{iR-v} and \eqref{norm-o}.

The following proof is divided into three steps.

{\bf Step 1. Prove that
\begin{align}
	\label{step1}\| \sqrt{\mu^*(x)} v^n\| = \lambda_n^{-p/4} o(1),\;\;  \|\sqrt{\mu^*(x)} \psi^n_x\| = \lambda_n^{-p/8} o(1).
	\end{align}}
	Taking the $L^2(0,\pi)$ inner product of \eqref{iR-v} with $\lambda_n^{-1}\mu^*(x) v^n$, we obtain
	\begin{align}\label{v-eq}
\rho\| \sqrt{\mu^*(x)}v^n\|^2 \le \gamma|\langle \phi^n_x,\; \lambda_n^{-1}\mu^*(x) v^n \rangle |
+ |\langle \beta\theta^n -T^n  ,\; \lambda_n^{-1}[\mu^*(x) v^n]_x  \rangle| +\lambda_n^{-p-1}o(1)\le \lambda_n^{-p/2} o(1).
\end{align}
Using $\|U^n\|_{\mathcal H}=1$ together with \eqref{norm-o}, \eqref{iR-eta-o}, \eqref{iR-u'-o}, and \eqref{bound}, we conclude the first estimation.

Next,  taking the $L^2(0,\pi)$ inner product of \eqref{iR-v} with $\lambda_n^{-1}\mu^*(x)\psi^n_x$ and using \eqref{norm}, \eqref{v-eq} yields
\begin{align}\label{psi-eq}
\beta\|\sqrt{\mu^*(x)} \psi^n_x\|^2 & \; \le |\langle i\rho\lambda_n v^n -T^n_{x} - \gamma\phi^n_x ,\;
\lambda_n^{-1} \mu^*(x)\psi^n_x\rangle| +\lambda_n^{-p-1}o(1)
\notag \\
& \; \le  |\langle T^n,\; \lambda_n^{-1}[\mu^*(x)\psi^n_{x}]_x\rangle |
+\lambda_n^{-1}\mu^*(0) |T^n(0) | |\psi^n_x(0) | +\lambda_n^{-p/4}o(1) .
\end{align}
Moreover, taking the $L^2(0,\pi)$ inner product of \eqref{iR-v} with $2\lambda_n^{-1}\mu^*(x)T^n$ gives
\begin{align*}
\mu^*(0) \lambda_n^{-1}|T^n(0)|^2
\le  \lambda_n^{-1}\big[   2|\langle i\rho\lambda_n v^n  - \gamma\phi^n_x +\beta\theta_x^n,\;   \mu^*(x)T^n \rangle| +  \|\sqrt{ \mu^*_x(x)}T^n \|^2+\lambda_n^{-p }o(1)\big]\le \lambda_n^{-p/2} o(1) ,
\end{align*}
where we use \eqref{iR-psi}, \eqref{norm-o}, \eqref{iR-eta-o}, \eqref{iR-u'-o}, and \eqref{v-eq}.
Using the above inequality, \eqref{norm}, \eqref{bound} and interpolation,  we obtain
$\lambda_n^{-1}   |T^n(0) | |\psi^n_x(0) | \le \lambda_n^{-1}   |T^n(0) | \|\psi^n_{xx}\|^{1\over2} \|\psi^n_x\|^{1\over2} \le \lambda_n^{-p/4} o(1).$
Substituting it into \eqref{psi-eq} and applying \eqref{iR-eta-o}--\eqref{bound} yields the second estimate in \eqref{step1}.

{\bf Step 2. Prove that
\begin{align}
	\label{step2}
	\|\sqrt{\mu^*(x)} \theta^n \| = \lambda_n^{-p/16} o(1),\;\;  \|\sqrt{\mu^*(x)} \phi^n_x\|= \lambda_n^{-p/32} o(1),\;\; \|\sqrt{\mu^*(x)} \varphi^n\| =  \lambda_n^{-p/64} o(1).
\end{align}
}
Taking the  $L^2(0,\pi)$ inner product of \eqref{iR-theta} with $\lambda_n^{-1}\mu^*(x)\theta^n$ and using \eqref{iR-u}, \eqref{iR-phi} leads to
\begin{align*}
a\|\sqrt{\mu^*(x)}\theta^n\|^2 \le  (\beta \|\mu^*(x) u^n_x\|+ d\|\mu^*(x)\phi^n\| )\|\theta^n\|+\lambda_n^{-1}  | \langle k\psi^n_{x} + m\phi^n_{x } ,\; (\mu^*(x)\theta^n)_x  \rangle|  + \lambda_n^{-1-p} o(1).
\end{align*}
Since $\lambda_n^{-1}\|\sqrt{\mu^*(x)} \theta^n_x\| = \lambda_n^{-p/8} o(1) $ by   \eqref{iR-psi} and \eqref{step1},   the first estimate in \eqref{step2} follows after combining the above inequality  with \eqref{norm-o}-\eqref{bound}.
The remaining two estimates are obtained analogously by taking the $L^2(0,\pi)$ inner product of \eqref{iR-theta} with $\mu^*(x) \phi^n$, \eqref{iR-varphi} with $\lambda_n^{-1}\mu^*(x) \varphi^n$, respectively.
%

{\bf Step 3. Prove that $\|U^n\|_{\mathcal H}=o(1)$, which  leads to a contradiction.}
Let $q(x)\in C^1[0,\tau]$ be a nonnegative function such that $q(0)=0, \; q(\tau) \not=0$. Taking the $L^2(0,\tau)$ inner product of (\ref{iR-v}),  (\ref{iR-varphi}),  (\ref{iR-theta}) with $2q(x)\mu^{-1}T^n, \; 2q(x)\phi_x^n, \;2q(x)\psi_x^n$, respectively, summing the results, and applying the estimates obtained in Steps 1-2, we obtain
\begin{align}
\label{global}
\begin{array}{l}
	-  2 \mu^{-1} \Re  [\langle i \lambda_n  \rho v^n  ,\; q(x)\zeta_x^n\rangle_{L^2(0,\tau)}-\beta \langle \theta_x^n  ,\; q(x)T^n\rangle_{L^2(0,\tau)}]+  q(\tau)\big[ \rho| v^n (\tau)|^2+   \mu^{-1} | T^n (\tau)|^2\\   \noalign{\medskip}
	+J |\varphi^n(\tau)|^2+ a  |\theta^n(\tau)|^2
	+b |\phi^n_x(\tau)|^2+ k |\psi^n_x(\tau)|^2
	+ 2m Re ( \psi^n_x(\tau)\overline{\phi^n_x(\tau)})\big]
	=  o(1).
	\end{array}\end{align}
	From   \eqref{iR-psi}, \eqref{iR-eta-o},  \eqref{iR-u'-o} and \eqref{step1},  we have $| \langle i \lambda_n v^n,\; q(x)\zeta_x^n\rangle_{L^2(0,\tau)} |,\;|\langle \theta_x^n  ,\; q(x)T^n\rangle_{L^2(0,\tau)}|  =o(1)  $ when $p\ge 8/5.$
	Substituting these into \eqref{global}, we obtain
	\begin{align}
		\label{zero1}
		\rho|v^n(\tau)|^2  +    | T^n (\tau)|^2 +J |\varphi^n(\tau)|^2+ a |\theta^n(\tau)|^2
		+b |\phi^n_x(\tau)|^2+ k |\psi^n_x(\tau)|^2
		+ 2m Re ( \psi^n_x(\tau)\overline{\phi^n_x(\tau)})
		=  o(1).
	\end{align}
	
	Next, taking the $L^2(\tau,\pi)$ inner product of (\ref{iR-v}),  (\ref{iR-varphi}),  (\ref{iR-theta}) with $2(\pi-x)u^n_x, \; 2(\pi-x)\phi_x^n, \;2(\pi-x)\psi_x^n$, respectively,   summing the results, and using \eqref{zero1},  we get
	\begin{align} \label{zer2} \begin{array}{l}
			\rho\| v^n\|^2_{L^2(\tau,\pi)} + \mu\| u^n_x\|^2_{L^2(\tau,\pi)} + J\| \varphi^n\|^2_{L^2(\tau,\pi)} + b\| \phi^n_x\|^2
			+ a\| \theta^n\|^2_{L^2(\tau,\pi)}
			+ k\| \psi^n_x\|^2_{L^2(\tau,\pi)}
			-\xi\| \phi^n\|^2_{L^2(\tau,\pi)}
			\\ \noalign{\medskip}   \displaystyle
			+2 Re \big[m \langle \psi_x^n,\;  \phi_x^n\rangle_{L^2(\tau,\pi)}
			+  \langle d\theta^n,\;  (\pi-x)\phi_x^n \rangle_{L^2(\tau,\pi)}- \langle d \varphi^n,\; (\pi-x)\psi_x^n \rangle_{L^2(\tau,\pi)} \big] =o(1).
	\end{array}\end{align}
	Using \eqref{iR-phi}, \eqref{iR-psi} and  \eqref{norm-o},   we estimate
	$
	\big|Re \big[  \langle  \theta^n,\; (\pi-x)\phi_x^n \rangle_{L^2(\tau,\pi)} - \langle   \varphi^n,\;(\pi-x)\psi_x^n \rangle_{L^2(\tau,\pi)} \big]\big|
	\le \big| i\lambda^n  \langle \psi^n, \; \phi^n\rangle_{L^2(\tau,\pi)}\big| + o(1) = o(1).
	$
	Substituting these into \eqref{zer2},  using \eqref{iR-eta-o}, \eqref{iR-u'-o}, \eqref{step1}, \eqref{step2},   and the assumption $bk>m^2$, we finally obtain $\|U^n\|_{\mathcal H}=0$, which contradicts \eqref{norm}.
\end{proof}

\section*{Acknowledgments}
The project is supported by the National Natural Science Foundation of China (grants No. 12271035, 12131008) and
Beijing Municipal Natural Science Foundation (grant No. 1232018).

%


Y.N. Sun, School of Mathematics and Statistics, Beijing Institute of Technology, Beijing, 100081, P.R. China

Email address: yanansun@bit.edu.cn

Q. Zhang, 
School of Mathematics and Statistics, Beijing Institute of Technology, Beijing, 100081, P.R. China

Email address: zhangqiong@bit.edu.cn
\end{document}